\documentclass[12pt]{amsart}

\usepackage{amscd}
\usepackage{verbatim}

\advance\hoffset by -0.5 truecm

\usepackage{amssymb}

\newtheorem{Theorem}{Theorem}[section]

\advance\hoffset by -0.5 truecm

\usepackage{amssymb}
\usepackage{graphicx}
\usepackage{subfigure}

\newtheorem{Proposition}[Theorem]{Proposition}

\def\V{\mbox{Var}}

\def\R\re

\def\V{\bf V}

\def\S{\bf{S}}

\def \re{{\mathbb R}}

\def \0{\lambda_{0}}

\begin{document}
\hyphenation{Ya-ma-be par-ti-cu-lar}
\title{Multiplicity of constant scalar curvature metrics on  warped products}

\author[J. M. Ruiz]{Juan Miguel Ruiz}
 \address{ENES UNAM \\
           37684 \\
          Le\'on. Gto. \\
          M\'exico.}
\email{mruiz@enes.unam.mx}



\begin{abstract}  
Let $(M^m,g_M)$ be a closed, connected manifold with positive scalar curvature and $(T^k,g)$ some flat $k$-Torus of  unit volume. By a result of F. Dobarro and E. Lami Dozo \cite{Dobarro}, there exists a unique $f: M \rightarrow \re_{>0}$ such that the warped product $M\times_f T^k$ has constant scalar curvature and unit volume.  We study the  Yamabe equation on these spaces.  We use techniques from bifurcation theory, along with spectral theory for warped products, to prove multiplicity results for the Yamabe problem. 
\end{abstract}

\maketitle

\section{Introduction}

Given a closed, Riemannian manifold $(N,h)$, the solution of the Yamabe problem (cf. in \cite{Parker}) gives a metric $\bar h$   for $N$, of constant scalar curvature and unit volume in the conformal class of $h$. These metrics are critical points of the Hilbert-Einstein functional restricted to conformal classes. The minima of the restricted functional are always realized. This was proved by the combined efforts of H. Yamabe \cite{Yamabe}, T. Aubin \cite{Aubin}, R. Schoen \cite{RSchoen}, and N. S. Trudinger \cite{Trudinger}, giving the solution of the Yamabe problem.  It is of interest to ask for other metrics of constant scalar curvature in the conformal class, as they are  critical points of the  Hilbert-Einstein functional on conformal classes, that are not necessarily minimizers.
Uniqueness of a metric of constant scalar curvature of unit volume in a conformal class is known to be true in some special cases: when the scalar curvature is non-positive, by the maximum principle; if $h$ is an Einstein metric that is not isometric to the round sphere,  by a result of M. Obata \cite{Obata}; and if the metric $h$ is close in the $C^{2,\alpha}$ topology to an Einstein metric and $dim(N)\leq7$ (or $dim(N)\leq24$  and $N$ is spin), by a result of  L. L. De Lima, P. Piccione and M. Zedda \cite{ Lima}. On the other hand, a rich variety of constant scalar curvature metrics that are not necessarily minimizers have been studied using bifurcation theory. Given a manifold $N$ of dimension greater than 2, let $\{g_t\}_{t\in [a,b]}$ be a path of metrics of constant scalar curvature and unit volume. We make use of  the notation in \cite{Lima} and call $t^* \in[a,b]$ a bifurcation instant in the path, if there exist   a sequence $t_n\in [a,b]$ and a sequence of Riemannian metrics $\bar g_n$ in the conformal class of $g_{t_n}$ ($\bar g_n \neq g_{t_n}$), of constant scalar curvature and unit volume, such that $\lim_{n\rightarrow \infty} t_n=t^*$ and $\lim_{n\rightarrow \infty} g_n=g_{t^*}$. We call the corresponding metric, $g_{t^*}$, a bifurcation point.
Bifurcation points in paths of constant scalar curvature metrics with unit volume have been studied recently, for example, on products of compact manifolds  in \cite{Lima}, on collapsing Riemannian submersions in \cite{Bettiol}, and on the product of a manifold of constant positive scalar curvature and a $k-Torus$ in \cite{Ramirez}.  Multiplicity of constant scalar curvature metrics have also been studied in the product manifold with a $k$-sphere \cite{Petean}, and, more generally, on sphere bundles in \cite{Otoba}.

Let $(M^m,g_M)$ be a closed, connected manifold with positive scalar curvature $R$ and $(T^k,g_t)$ some flat $k$-Torus of  unit volume. In this article we will be interested in finding multiplicity of constant scalar curvature metrics in warped products  $M^m\times_f T^k$, with the use of  bifurcation theory. Namely, we find bifurcation points in paths of metrics with unit volume and constant scalar curvature on $M^m\times_f T^k$, for fixed $g_M$ and fixed weight $f$.

Metrics of constant scalar curvature on the warped product of a compact manifold with a compact manifold of flat scalar curvature were studied  in \cite{Dobarro} by F.  Dobarro and E. Lami Dozo. It is shown there that, under these conditions, there is a unique weight $f$ that makes the warped product of constant scalar curvature and unit volume. Moreover, this weight does not depend on the metric of the second factor, as long as said metric is  flat and of unit volume.

Specifically, by letting $u=f^\frac{k+1}{2} \in C^{\infty}$,  the equation

$$\frac{-4k}{k+1} \Delta_M u + R u=  S u,$$
\noindent holds for $(M^m\times T^k,g_M+f^2g_t)$ (Theorem 2.1 in \cite{Dobarro}), where $S$ is the resulting scalar curvature on the warped product. A manifold of constant scalar curvature results if $S$ is constant. This leads to the linear eigenvalue problem 
$$\square \ u=  S u$$
on $M$, where $\square= \frac{-4k}{k+1} \Delta_M  + R $ and $S \in \re_{>0}$. This eigenvalue problem has a unique nonnegative solution $u_1$ such that $max\{u_1\}=1$ and $ S$ satisfies
\begin{equation}
\label{dozo}
S = \inf \{\int_M \left(\frac{4k}{k+1}|\nabla v|^2+R v^2\right)dV; v\in H^1(M), \int_M v^2 dV=1\}
\end{equation}
\noindent where $H^1(M)$ is the Sobolev space of order 1. The infimum of equation (\ref{dozo}) is realized by a constant multiple of $u_1$. Thus the weight we need is $f=u_1^\frac{2}{k+1} \in C^{\infty} $, normalized, so that $(M^m\times T^k,g_M+f^2 g_t)$ has unit volume   (we refer the reader to \cite{Dobarro}). It follows from equation (\ref{dozo}) that the constant $ S$ and the weight $f$ are independent of  $(T^k,g_t)$, as long as $g_t$ is flat, of unit volume and the dimension is $k$. Also, since we started with a positive scalar curvature $R$ for $M$, the constant $S$ will also be positive.

Our strategy will be to vary the metric $g_t$ on  $T^k$ to get a family of warped product metrics of constant scalar curvature and unit volume, and then prove the existence of a bifurcation point in this family.  We use this strategy, to generalize some recent results by Ramirez-Ospina in \cite{Ramirez}, on multiplicity of constant scalar curvature metrics on the product  of  a manifold with constant, positive scalar curvature, and a flat $k$-torus, to more general warped products of these manifolds. 

The spectrum of the warped product of compact manifolds, $M\times_f F$, was studied thoroughly by N. Eijiri in \cite{Eijiri}. It is shown there that given a base manifold, $(M,g_M)$, and a fiber $(F,g_F)$ with eigenvalues  $0=\lambda_0<\lambda_1\leq...\leq \lambda_i \rightarrow \infty$, then, the eigenvalues of $\Delta_{M\times_f F}$ are the list $\{\mu_i^j\}$ for $i,j=0,1,2,...$, where, for each $i$, the  $\{\mu_i^j\}$ are the eigenvalues for the operator 

\begin{equation}
\label{operator}
L_{\lambda_i}= \Delta_M-\frac{k}{f} \nabla_{grad \  f}+\frac{\lambda_i}{f^2}.
\end{equation}

\noindent Note that when $\lambda_0=0$, the operator is
\begin{equation}
\label{zerooperator}
L_{0}= \Delta_M-\frac{k}{f} \nabla_{grad \  f},
\end{equation}

\noindent which does not depend on the metric $g_t$ at all. 

We use bifurcation theory and combine the study of the operator $L_{\lambda}$  with the study of the spectrum of  $T^k$ to obtain the following result.

\begin{Theorem}
	\label{thm}
	Let $(M^m,g_M)$ be a closed Riemannian manifold with positive scalar curvature.  Consider the  warped products $M^m\times_f T^k$  of constant scalar curvature $S$ and unit volume,  of $g_M$ with any flat metric of unit volume on $T^k$.  Let $L_0$ be the elliptic operator given by equation (\ref{zerooperator}). Then, if $\frac{S}{m+k-1} \notin Spec\{L_0\}$, there are infinitely many bifurcation points for families of the form $M^m\times_f T^k$. Moreover, we may parametrize a set of bifurcation points with a codimension 1 submanifold of the space $\mathcal{M}$ of flat  metrics with unit volume on $T^k$.
\end{Theorem}

We remark that the condition $\frac{S}{m+k-1} \notin Spec\{L_0\}$ is necessary for the construction of paths with bifurcating points, when we vary only the second factor. Consider for example the manifold $(M^m,g_M)= (\S^2\times T^2, g_1+g_2)$, where $g_1$ is the round metric and $g_2$ is a flat metric on $T^2$ such that $\lambda_1= \frac{2}{5}$, and the product has unit volume (see section 4 for details on the construction of such metrics on $T^k$). Then, for any flat metric $g_3$ on $T^2$, we have that $\lambda= \frac{2}{5}$ is also an eigenvalue of  $(\S^2\times T^2\times T^2, g_0+g_2+g_3)$  and then we get
$\lambda- \frac{Scal}{4+2-1}=0$, therefore losing a necessary condition (see (1) and (2) of Proposition \ref{Zed}) to ensure the existence of paths with bifurcation points. Small perturbations of this example  illustrate that this is   also  the case for  warped products.

Of course, if the scalar curvature $R$ of the manifold $(M^m,g)$ is constant and of unit volume from the beginning, then we have that  $f=1$ (see equation (\ref{dozo})) and we recover the case of the product manifold, studied by H. F. Ramirez-Ospina in \cite{Ramirez}. We remark that in this case,  condition $\frac{R}{m+k-1} \notin Spec\{\Delta_M\}$ is  also necessary for the construction of paths with bifurcation points, where we vary the metric on $T^k$ as discussed above.


\textit{Acknowledgments:} The author would like to thank Prof. Jimmy Petean for many helpful discussions on the subject and many useful comments on the first draft of this paper. This work was supported by program UNAM-DGAPA-PAPIIT IA106116.


\section{Some spectral theory for warped products}

Given a compact manifold $(M,g_M)$ of positive scalar curvature, we  fix the unique weight $f$ that makes the warped product  $(M\times T^k,g_M+ f^2g)$ of constant scalar curvature and unit volume ($g$ is any flat, unit volume, metric of $T^k$)  and study its spectra.

The eigenvalues $\alpha_i$ of the Laplacian of a compact  Riemannian manifold are discrete, of finite multiplicity and accumulate only at infinity: $\alpha_i \rightarrow \infty$ as $i\rightarrow \infty$. For the warped product $(M\times T^k,g_M+ f^2 g)$, the eigenvalues are given by the following procedure:

We  first denote the eigenvalues of $(T^k,g)$, by $0=\lambda_0<\lambda_1\leq...\leq \lambda_i \rightarrow \infty$. Then, for each $i=0, 1,2,...$, we define the operator

\begin{equation}
\label{operator}
L_{\lambda_i}= \Delta_M-\frac{k}{f} \nabla_{grad \  f}+\frac{\lambda_i}{f^2},
\end{equation}

\noindent and denote its eigenvalues by $\mu_i^0<\mu_i^1\leq...\leq \mu_i^j \rightarrow \infty$. The eigenvalues of $\Delta_{M\times T^k}$ are exactly the list $\{\mu_i^j\}$ for $i,j=0,1,2,...$ (we refer the reader to the work of N. Eijiri \cite{Eijiri} for details).

We recall the following results by K. Tsukada.
 
\begin{Proposition} \label{eigenwarped} \emph{(Theorem 2, in \cite{Tsukada})}
	For a warped product of compact, connected manifolds $(M^m\times T^k,g_M+ f^2 g)$, we have $\mu_i^0 \  \frac{||f||_k^2}{||1||_k^2}\leq \lambda_i$, for $i\in \mathbb{N}$. Equality holds if and only if $ f$ is constant.
\end{Proposition}


For fixed $\lambda$, we denote the eigenvalues of the operator $L_{\lambda}=L(\lambda)$ defined in equation (\ref{operator}), by $\mu^0(\lambda)\leq \mu^1(\lambda)\leq...\leq \mu^j(\lambda) \leq...$, for $j\in \mathbb{N}$. In this sense, $\mu^0$ is increasing as a function of $\lambda$: 

\begin{Proposition} \label{ineq} \emph{(Lemma 1, in \cite{Tsukada})}
Let $0\leq \lambda \leq \lambda'$, then $\mu^0(\lambda)\leq \mu^0(\lambda')$. Equality holds if and only if $ \lambda = \lambda'$.
\end{Proposition}

It will be useful for our purposes, to generalize Proposition \ref{ineq} to any eigenvalue $\mu^j(\lambda)$, not just the first one. For this, we consider the min-max characterization of $\mu^j$.
 
  Let $Gr_j(C^{\infty}(M))$ be the $j-$dimensional Grassmanian in $C^{\infty}(M)$. Recall that
$$\mu^j(\lambda)= \inf_{V \in Gr_j(C^{\infty}(M))} \sup_{v\in V \setminus \{0\}} \frac{\langle L_{\lambda} v,v \rangle}{\langle v,v\rangle}$$

With this characterization we are able to generalize Proposition \ref{ineq} to any eigenvalue $\mu^j(\lambda)$ of $L_{\lambda}$. 
\begin{Proposition}
	\label{minmax}
Let $0\leq \lambda \leq \lambda'$, then  $\mu^j(\lambda)\leq \mu^j(\lambda')$, for $j\in \mathbb{N}$. Equality holds if and only if $ \lambda = \lambda'$.
\end{Proposition}

\begin{proof}
First note that
$$ \int_M ( f^k v) (\Delta_M v) \ dV_g$$
\begin{equation}
\label{gauss}
=\int_M  (\nabla v, \nabla (f^k v))_{g} dV_g= \int_M f^k (\nabla v, \nabla v)_{g} dV_g+\int_M k f^{k-1}v (\nabla v, \nabla f)_{g} dV_g.
\end{equation}
We compute 


$$\langle L_{\lambda} v,v \rangle= \int_M v (\Delta_M v- \frac{k}{f}(\nabla v, \nabla f)_{g}+\frac{\lambda}{f^2}v)  f^k dV_g,$$

$$=\int_M \left(f^k v \Delta_M v- k f^{k-1}  v (\nabla v, \nabla f)_{g}+\lambda f^{k-2}  v^2\right) dV_g,$$
\noindent and using  equation (\ref{gauss}),
$$\langle L_{\lambda} v,v \rangle= \int_M f^k (\nabla v, \nabla v)_{g} dV_g+\int_M k f^{k-1}v (\nabla v, \nabla f)_{g} dV_g- \int_M k f^{k-1}v (\nabla v, \nabla f)_{g} dV_g$$
$$+ \ \lambda  \int_M f^{k-2}  v^2 dV_g$$
$$=\int_M f^k || \nabla v||_g^2 dV_g+ \lambda \int_M f^{k-2} v^2 dV_g.$$
\noindent Since $f$ is positive, this implies that 
$$\langle L_{\lambda} v,v \rangle\leq \langle L_{\lambda'} v,v \rangle,$$
\noindent if  $0\leq \lambda\leq \lambda'$. Also, $\langle L_{\lambda} v,v \rangle= \langle L_{\lambda'} v,v \rangle$, if and only if  $ \lambda= \lambda'$.

Finally, since for each $j \in \mathbb{N}$,

$$\mu^j(\lambda)= \inf_{V \in Gr_j(C^{\infty}(M))} \sup_{v\in V \setminus \{0\}} \frac{\langle L_{\lambda} v,v \rangle}{\langle v,v\rangle},$$
then, 
$\mu^j(\lambda)\leq \mu^j(\lambda')$,  if  $0\leq \lambda\leq \lambda'.$ And  $\mu^j(\lambda)= \mu^j(\lambda')$,  if  and only if $\lambda=\lambda'.$

\end{proof}

\section{Bifurcation of solutions for the Yamabe problem}

We now take a look at the bifurcation of solutions for the Yamabe problem. Consider a path of metrics $\{g_t\}_{t\in [a,b]}$, of constant scalar curvature and unit volume for a fixed closed manifold $N$ ($dim(N\geq3)$). We use the notation in \cite{Lima} and call $t^* \in[a,b]$ a bifurcation instant for the path if there exist   a sequence $t_n\in [a,b]$ and a sequence of Riemannian metrics $\bar g_n$ in the conformal class of $g_{tn}$ ($\bar g_n \neq g_{tn}$), of constant scalar curvature and unit volume, such that $\lim_{n\rightarrow \infty} t_n=t^*$ and $\lim_{n\rightarrow \infty} g_n=g_{t*}$. We will call the corresponding metric, $g_{t^*}$, a bifurcation point.

Our result will be an application of the bifurcation  of solutions for the elliptic equation of constant scalar curvature. The following is due to De Lima, Piccione and Zedda, \cite{Lima}.

\begin{Proposition}\emph{(Theorem 3.3, in \cite{Lima})}
\label{Zed}
Let $N$ be a compact manifold with $dim (N)\geq3$, and $ \{g_{t}\}_{t \in [a,b]}$ a $C^{1}$-path of Riemannian metrics on $N$ having constant scalar curvature and unit volume. For each $t\in[a,b]$, let $S_t$ denote the scalar curvature of $g_t$ and $n_t$, the number of eigenvalues of the Laplace-Beltrami operator, $\Delta_{g_t}$, counted with multiplicity, that are less than  $\frac{S_t}{dim(N)-1}$.
Suppose that the following are satisfied,
\begin{enumerate}
\item $\frac{S_a}{dim(N)-1}=0$ or $\frac{S_a}{dim(N)-1}\notin Spec\{\Delta_{g_a}\}$.
\item $\frac{S_b}{dim(N)-1}=0$ or $\frac{S_b}{dim(N)-1}\notin Spec\{\Delta_{g_b}\}$.
\item $n_a \neq n_b$.
\end{enumerate}
Then there exists a bifurcation instant $t_*\in  (a,b)$ for the family of metrics $ \{g_{t}\}_{t \in [a,b]}$. 

\end{Proposition}

\noindent Often, $n_t$ is referred also as the Morse index of $g_t$ (as a critical point of the Hilbert-Einstein functional). In this sense, if either $\frac{S_t}{dim(N)-1}=0$ or $\frac{S_t}{dim(N)-1}\in Spec\{\Delta_{g_t}\}$, it is said that $g_t$ is degenerated.

Following  Proposition \ref{Zed}, our strategy will be  to construct the $C^{1}$-path of Riemannian metrics   $(M^m\times T^k,g_M+ f^2 g_t)$, having constant scalar curvature and unit volume. We then prove that the three conditions mentioned in Proposition \ref{Zed} are satisfied, thus assuring the existence of bifurcation points in the paths constructed. 

\section{Proof of Theorem \ref{thm}}

We are now ready to prove Theorem \ref{thm}.
\begin{proof}

We first recall some elementary facts on the eigenvalues of the Laplacian of flat, unit volume metrics,  on $T^k$; we refer the reader, for example, to \cite{Berger}. Recall that a flat metric on $T^k$ is given by $\re^k/\Gamma$, where $\Gamma$ acts by isometries on $\re^k$ and the lattice, $\Gamma\subset \re^k$, is a set consisting of the integral linear combinations of a basis of $\re^k$. We may thus associate to the lattice $\Gamma$ a matrix $B\in GL(k)$, with its columns given by a basis $(v_1, v_2,...,v_k)$ of $\re^k$.  Moreover, since we are working with metrics of unit volume, we will consider only basis $B$, such that $|det(B)|=1$. Recall also that if  two basis are isometric in $\re^n$, then the lattices associated to them, and thus the resulting metrics of $T^k$, are the same. 

Now, given a lattice $\Gamma$, the eigenvalues of the Laplacian of the Riemannian metric for $T^k=\re^k/\Gamma$, are given by $\lambda_i=4 \pi^2||\beta_i||^2$, where $\beta_i \in \Gamma^*$, and $\Gamma^*$ is the lattice dual to $\Gamma$. The dual lattice is the lattice associated to the dual of the basis associated to $\Gamma$. In practice, this means that if $B$ is a matrix of a basis associated to $\Gamma$, then the inverse of the transpose, $(B^T)^{-1}$, is a matrix of a basis associated to $\Gamma^*$. 
 In the following, we will denote by $\mathcal M$ the set of  flat metrics of unit volume on $T^k$. Given a metric $g_t$ in $\mathcal M$, we will denote by $\lambda_{i}(t)$ its eigenvalues and by $G_t$ the warped product metric   $G_t=g_M+f^2 g_t$. Also, we will let $\bar S=\frac{S}{m+k-1}$ (recall that $f$ is fixed and that $S$ is the positive, constant, scalar curvature of $(M\times T^k,G_t)$; neither, as argued above,  depend on $t$).
Given a metric $g_1 \in \mathcal{M}$ with associated matrix $B=(v_1,v_2,...,v_k)$, we denote  its dual matrix by 
 $(B^T)^{-1}$, with columns $( w_1,w_2,..., w_k)$, i.e. $(B^T)^{-1}=( w_1,w_2,..., w_k)$. With this in mind, for $t\geq1$, consider the path of flat, unit volume metrics $g_t\in \mathcal{M}$, given by  the family of matrices
\begin{equation}
\label{eigenlambda1}
 B_t=(t^{k-1} v_1,\frac{1}{t} v_2,...,\frac{1}{t}v_k),
 \end{equation}

\noindent (hence $(B_t^T)^{-1}=(\frac{1}{t^{k-1}} w_1,t w_2,...,t w_k)$). Consider also the corresponding path $\{G_t\}_{t\geq1}$, with $G_t=g_M+f^2g_t$. 

 We now show that there exist $a,b\in \re_{>0}$ such that the path $\{G_t\}_{t \in [a,b]}$, contains a bifurcation point.  Recall from Proposition \ref{Zed}, that we want to start our path of metrics $\{G_t\}_{t\in [a,b]}$, with a metric $G_a$ such that $\bar S\notin Spec\{\Delta_{G_a}\}$.
 If $\bar S\notin Spec\{\Delta_{G_1}\}$, we let $G_a=G_1$, otherwise we do the following: 
 
 Suppose that $\bar S\in Spec\{\Delta_{G_1}\}$. Let  $\mu_{i_1}^{j_1}, \mu_{i_2}^{j_2},...,\mu_{i_n}^{j_n}$ be the finite set of eigenvalues of the Laplacian of $G_1$, equal to $\bar S$. Recall the hypothesis of Theorem \ref{thm}:  $\bar S \notin Spec\{L_0\}$;
 this means that neither of $i_1,i_2,...,i_n$ is equal to zero, for the given set of eigenvalues. Thus, by construction, advancing the path (i.e. increasing $t$)  modifies the values of each of the  eigenvalues of the Laplacian of the metric on $T^k$, in particular, of $\lambda_{i_1}, \lambda_{i_2},...,\lambda_{i_n}$. By Proposition \ref{minmax}, this  in turn means that the values of  $\mu_{i_1}^{j_1}, \mu_{i_2}^{j_2},...,\mu_{i_n}^{j_n}$, are modified, as they are increasing functions of $\lambda_{i_1}, \lambda_{i_2},...,\lambda_{i_n}$. Since the eigenvalues are discrete and the path is continuous, there exists in the path a metric $G_a$ ($a>1$) such that  $\bar S\notin Spec\{\Delta_{G_a}\}$.

Recall that  each metric $g_t$, on the path $\{g_t\}_{t\geq1}$,  has as a  subset of its set of eigenvalues of its Laplacian, the list $\{\lambda_{i_q}\}_{q\in \mathbb{N}} \subset \{\lambda_{i}\}_{i\in \mathbb{N}} $, given by  
\begin{equation}
	\label{eigenlambda}
	\lambda_{i_q}(t)=4\pi^2 \frac{1}{ (t^{(k-1)})^2} q^2 ||w_1||^2,
\end{equation}
\noindent $q \in \mathbb N$ (namely, those associated with the integer multiples of $ \frac{1}{ t^{(k-1)}} w_1$ in $\Gamma^*$). 

Thus, for the corresponding path $\{G_t\}_{t\geq1}$, equation (\ref{eigenlambda}) and Proposition \ref{eigenwarped} imply that for  $q\in\mathbb N$ we have
\begin{equation}
\label{eigenlambda3}
\mu_{i_q}^0 (t)  \leq  \frac{||1||_k^2}{||f||_k^2 } \  4\pi^2 \frac{1}{ (t^{(k-1)})^2} q^2 ||w_1||^2.
\end{equation}

Now, let $n_a$ be the Morse index of $G_a$, that is, the number of eigenvalues of the Laplacian of $G_a$, counted with multiplicity, such that they are less than $\bar S$. Let $q_1$ be a positive integer such that  $q_1>n_a$. Then, we advance our path $\{G_t\}_{t\geq a}$, starting from $G_a$, until $G_{t_1}$, where $t_1\in\re_{>0}$ is big enough so that 

$$     \frac{||1||_k^2}{||f||_k^2}  \  \left( 4\pi^2 q_1^2||w_1||^2\frac{1}{ t_1^{2(k-1)}}\right)  <\bar S.$$
 Hence, according to  equation (\ref{eigenlambda3}), we have
	$$\mu_{i_q}^0 (t_1)  <\bar S,$$
for every $q\in \mathbb N$, $q \leq q_1$. That is,  $G_{t_1}$ has at least these $q_1$ eigenvalues that are less than $\bar S$. This makes the Morse index of  $G_{t_1}$, strictly greater than that of $G_{a}$. To finish the path,  we must find  a final metric $G_b$, $b \in (t_1,t_1+\epsilon)$, $\epsilon>0$, such that $\bar S\notin Spec\{\Delta_{G_b}\}$, while keeping its Morse index greater than $n_a$. We achieve this by advancing our  path a little more, as we did in the case of $G_a$.  

Note that, by construction, advancing the path modifies all the eigenvalues of the Laplacian of $g_{t_1}$, as we increase $t$. This in turn modifies those eigenvalues of the Laplacian of $G_{t_1}$ that were equal to $\bar S$, since, by Proposition \ref{minmax}, the eigenvalues $\mu^{j}(\lambda)$ are strictly increasing, as functions of $\lambda$. This means that if we advance the path, the eigenvalues $\lambda_{i_{1}}(t), \lambda_{i_{2}}(t),...,\lambda_{i_{q_1}}(t)$ would decrease, and so the eigenvalues  $\mu_{i_1}^0(t), \mu_{i_2}^0 (t),...\mu_{i_{q_1}}^0(t) $ would also decrease (Proposition \ref{minmax}); this means that we would still have that $\mu_{i_1}^0(t), \mu_{i_2}^0(t) ,...\mu_{i_{q_1}}^0(t)$  are strictly less than $\bar S$, for $t>t_1$. Since the path is continuous and the eigenvalues are discrete, there exists a metric $G_b=g_M+f^2g_b$, for $b\in (t_1,t_1+\epsilon)$, $\epsilon>0$, such that $\bar S\notin Spec\{\Delta_{G_b}\}$.  That is, the Morse index of $G_b$ is at least $q_1$ ($q_1>n_a)$ and  $\bar S\notin Spec\{\Delta_{G_b}\}$.

We have constructed a path of metrics, $\{G_t\}_{t\in[a,b]}$,  that satisfies the conditions of Proposition \ref{Zed}, proving thus the existence of a bifurcation point somewhere in the path. 


For the second part of the Theorem, we note that we may construct a whole submanifold of $\{g_M\}\times \mathcal{M}$  that parametrizes a set of bifurcation points. We achieve this by constructing a tubular neighborhood around a path $\{G_t\}_{t\in[a,b]}$, that satisfies the conditions of Proposition \ref{Zed},  like in the first part of the proof.

As before, given a metric $g_1\in \mathcal{M}$, we let $B=(v_1, v_2,...,v_k)$, and $(B^T)^{-1}=(w_1,w_2,...,w_k)$ be some associated matrices to the lattices $\Gamma$ and $\Gamma^*$, respectively. Without loss of generality, we may assume that the first column, $w_1$, is the smallest column of $(B^T)^{-1}$ (recall that isometries of the basis $B$ generate the same lattice $\Gamma$).  Let  $C=||w_1||^2$, and consider the submanifold  $\mathcal{N}\subset \mathcal{M}$, of metrics $g$ with associated $\Gamma_g^*$, such that the smallest column of $(B_g^t)^{-1}$ satisfies $||(w_1)_g||^2=C$. Consider the submanifold $V$ given by the intersection of $\mathcal{N}$ and a neighborhood of $g_1$. Note that if we construct a path for each $g\in V$, using equation (\ref{eigenlambda1}), they will not intersect, for any $t\geq1$, and, moreover, will form a tubular neighborhood of the path for $g_1$. Arguing as in the first part of the proof, for each metric $g\in V$, its corresponding path, $\{G_t\}_{t\geq1}$, will have at least one bifurcation point.
Hence, since the paths do not overlap, the submanifold $V\subset \mathcal{M}\approx \{g_M\}\times \mathcal{M}$, gives us the desired submanifold that parametrizes a set of bifurcation points.

\end{proof}

We remark, as  in the introduction, that when the scalar curvature of $(M,g)$ is constant and of unit volume, we have $f=1$ (see equation \ref{dozo} or the proof of Theorem 3.1 in \cite{Dobarro}) and thus recover the product case, studied in \cite{Ramirez}.

We finally note that if we are to vary only the metric of the second factor, the approach of bifurcation theory does not work if $\bar S  \in Spec\{L_0\}$ (not even on the product case, $f=1$). Indeed,
for $\lambda_0=0$, we have that the operator
$$L_0= \Delta_M-\frac{k}{f} \nabla_{grad \  f},$$
\noindent does not depend on $(T^k,g_t)$ at all.
Thus,  $\bar S $ would be in  $ Spec\{L_0\}$, for all $t$, regardless of the chosen path $(M\times T^k,g_M+f^2 g_t)$, $t>0$. This in turn implies that the metrics cannot satisfy conditions (1) and (2) of Proposition \ref{Zed}, for any $t>0$, and thus there is no guarantee of the existence of bifurcation points.

\end{document}